\documentclass[11pt,reqno]{amsart} 
\usepackage{amsmath,amssymb,rawfonts}
\usepackage{tikz}
\usetikzlibrary{matrix,arrows}

\newcommand{\be}{\begin{equation}}
\newcommand{\ee}{\end{equation}}
\newtheorem{theorem}{Theorem}[section]

\newtheorem{defi}[theorem]{Definition}
\newtheorem{prop}[theorem]{Proposition}

\newtheorem{remark}[theorem]{Remark}
\newtheorem{ex}[theorem]{Example}
\newtheorem{com}[theorem]{Comments}
\begin{document} 
\title{SHEAVES IN ELEMENTARY MATHEMATICS:\\ T\lowercase {he case of positive integer numbers}}
\author{Joaqu\'in Luna-Torres}
\thanks{ {\bf --------------------------------}\\ The author would like to thank  Fernando Zalamea, professor of {\it Universidad Nacional de Colombia}, for his comments and  many valuable suggestions that led to an improvement of this work.}
\dedicatory{Universidad Distrital Francisco Jos\'e de Caldas}

\email{jluna@ima.usergioarboleda.edu.co}
\begin{abstract}
We aim to use the concept of ``sheaf"  to establish a link between certain aspects of the set of positive integers numbers, a topic corresponding to the elementary mathematics, and some fundamental ideas of  contemporary mathematics.

We hope that this type of approach helps the school students to restate some problems of elementary mathematics in  an environment deeper and suitable  for its study.
\end{abstract}
\maketitle 
\baselineskip=1.7\baselineskip
\section{ Introduction}
In this paper, we aim to use the concept of ``sheaf"   to establish a link between certain aspects of the set of positive integers numbers, a topic corresponding to the elementary mathematics, and some fundamental ideas of  contemporary mathematics.

We hope that this type of approach helps the school students to restate some problems of elementary mathematics in an  environment deeper and suitable  for its study.

The paper is organized as follows: We describe, in section $2$, some orders on the set of positive integers $\mathbb Z^{+}=\{1, 2, 3, \cdots\}$ and we introduce the basic framework on partially ordered sets. In section $3$, following  S. MacLane and  I. Moerdijk  (see \cite{MM}), we present the concepts of lattices, Heyting algebras and categories; after that, in section $4$ we construct some  Grothendieck topologies on  the category $\mathbb Z^{+}_D$ and its sheaves,then  in section $5$ we study  some  subobject classifiers for the sites $(\mathbb Z^{+}_D, J)$.  Finally in section $6$ we present some categories equivalent to $\mathbb Z^{+}_D$

\section{Preliminaries}
In this section, our main aim is to describe some orders on the set of positive integers $\mathbb Z^{+}=\{1, 2, 3, \cdots\}$.
\begin{defi}
The multiplicative partial ordering on $\mathbb Z^{+}$ is defined as follows: for all $a,b\in \mathbb Z^{+}$,\ $a\prec b$ if and only if there exists $c\in \mathbb Z^{+}$ such that $a=bc$.
\end{defi}
From now on, $\mathbb Z^{+}_{M}$ denotes the partially ordered set $(\mathbb Z^{+}, \prec)$.
\begin{defi}
The divisibility  partial ordering on $\mathbb Z^{+}$ is defined as follows: for all $k,n\in \mathbb Z^{+}$,\ $k\leqslant n$ if and only if there exists $t\in \mathbb Z^{+}$ such that $n=kt$.
\end{defi}
We will denote by $\mathbb Z^{+}_{D}$ the partially ordered set $(\mathbb Z^{+}, \leqslant)$.

From A. J. Lindenhovius  (\cite{AL}), we take the following ideas
\begin{defi}
Let $(P,\leqslant)$ be a partially ordered set and $M \subseteq P$. We say that $M$ is an up-set if for each $x \in M$ and $y \in P$ we have $x \leqslant y$ implies $y \in M$. Similarly, $M$ is called a down-set if for each $x \in M$ and $y\in P$ we have $y\leqslant x$ implies $y \in M$. 

Given an element $x \in P$, we define the up-set and down-set generated by $x$ by $\uparrow x = \{y \in P : x\leqslant y\}$ and
$\downarrow x = \{y \in P : y \leqslant x\}$, respectively.

We can also define the up-set generated by a subset $M$ of $P$ by\linebreak $\uparrow M = \{x \in P : m \leqslant x \,\ \text{for some}\,\ m \in M\} =\bigcup_{m\in M} \uparrow m$, and similarly, we define the down-set generated by $M$ by $\downarrow M =\bigcup_{m\in M} \downarrow m$.

 We denote the collection of all up-sets of a partially ordered set $P$ by $\mathcal U(P)$ and the set of all down-sets by $\mathcal D(P)$. 
\end{defi}

For $Z^{+}_{D} :=(\mathbb Z^{+}, \leqslant)$, we now observe that
\begin{itemize}
\item The down-set generated by $n\in \mathbb Z^{+}$ is $\downarrow n = \{k \in \mathbb Z^{+} : k \leqslant n\}= \mathfrak D_n$, the set of all divisors of $n$.
\item If $M\subseteq  \mathbb Z^{+}$ then the down-set generated by $M$ is $\downarrow M =\bigcup_{m\in M} \mathfrak D_m$.
\item The up-set generated by $n\in \mathbb Z^{+}$ is $\uparrow n = \{k \in \mathbb Z^{+} : n \leqslant k\}= \mathfrak M_n$, the set of all multiples of $n$.
\item If $P\subseteq  \mathbb Z^{+}$ then the up-set generated by $P$ is $\uparrow M =\bigcup_{p\in P} \mathfrak M_p$.
\end{itemize}

\section{Lattices and categories}

Using  S. MacLane and  I. Moerdijk  results (see \cite{MM}), we have that a lattice $L$ is a partially ordered set which, considered as a category, has all binary products and all binary coproducts. In this way, if we write $m, n, k$ for objects of $L= \mathbb Z^{+}_{D}$ then $m\leqslant n$ if and only if there is a unique arrow $m\rightarrow n$, the coproduct of $m$ and $n$ (i.e. their pushout) is the least upper bound, or least common multiple, $m\vee n= lcm\{m,n\}$ and the product (i.e. their pullback) is the greatest lower bound, or greatest common divisor, $m\land n=gcd\{m,n\}$. 
It is clear that $\mathbb Z^{+}_{D}$ is a distributive lattice.

On the other hand, for each object $n$ in $\mathbb Z^{+}_{D}$, the set $\mathfrak D_n$  of all divisors of $n$ is a Heyting algebra:
the exponential $n^k$, for $n,k \in\mathbb Z^{+}$, is usually written as $k\Rightarrow n$; by its definition it is characterized by the adjunction 
\be
t\leqslant (k\Rightarrow n)\,\,\ \text{ if and only if}\,\,\  gcd\{t,k\}\leqslant n.
\ee
In other words, $k\Rightarrow n$ is the least common multiple for all  those elements $t$ with  $gcd\{t,k\}\leqslant n$.

In any Heyting algebra we define the negation of $x$ as
\be
\neg x= (x\Rightarrow 0).
\ee

Consequently for $n$ in $\mathbb Z^{+}_D$
\be
\neg n= (n\Rightarrow 1)= lcm\{ t\in \mathbb Z^{+}\mid gcd\{t,n\}=1\}.
\ee
Some of the familiar properties of negation still apply, as follows.
\begin{prop}
In the Heyting algebra $\mathfrak D_n$,
\begin{enumerate}
\item[(i)] $n\leqslant \neg\neg n$,
\item[(ii)] $n\leqslant k$ implies $\neg k\leqslant \neg n$,
\item[(iii)] $\neg n=\neg\neg\neg n$,
\item[(iv)] $\neg\neg \ mcd\{ n,k\} = mcd\{ \neg\neg n, \neg\neg k\}$.
\end{enumerate}
\end{prop}
\begin{proof}
See S. MacLane and  I. Moerdijk  (\cite{MM} pp 53-54)
\end{proof}
\begin{com}\
\begin{itemize}
\item The category $\mathbb Z^{+}_D$ is the opposite of the category  $\mathbb Z^{+}_M$; they are really the same thing, just as a locale is the same thing as a frame (see P. T. Johnstone \cite{PJ}).
\item Using the Fundamental Theorem of Arithmetic it is easy to show  that  if $n\in \mathbb Z^{+}$ is the product of a finite number of distinct primes then $\mathfrak D_n$ is a Boolean Algebra.
\end{itemize}
\end{com}

\section{Grothendieck topologies on  $\mathbb Z^{+}_D$}
In order to construct Grothendieck topologies, we first define sieves.
\begin{defi}
Given an element $k$ in $\mathbb Z^{+}$, a subset $S$ of $\mathbb Z^{+}$ is called a sieve on $k$ if $S\in \mathcal D(\mathfrak D_k)$.
\end{defi}
\begin{defi}
A Grothendieck topology  on  the category $\mathbb Z^{+}_D$ is a function $J$ which assigns to each object $n$ of  $\mathbb Z^{+}$ a collection $J(n)$ of sieves on $\mathbb Z^{+}_D$, in such a way that
\begin{enumerate}
\item[(i)] the maximal sieve $\mathfrak D_n$ (the set of all divisors of $n$) is in $J(n)$;
\item[(ii)] (stability axiom) if $S\in J(n)$ and $k\leqslant n$ then $S\cap\mathfrak D_k$ is in $J(k)$;
\item[(iii)] (transitivity axiom) if $S\in J(n)$ and $R$ is any sieve on $n$ such that $R\cap \mathfrak D_k$ is in $J(k)$ for each $k\in S$, then $R\in J(n)$.
\end{enumerate}
\end{defi}
\begin{ex}\
\begin{itemize}
\item The trivial Grothendieck topology on $\mathbb Z^{+}_D$ is given by $J_{tri}(n) = \{\mathfrak D_n\}$.
\item The discrete Grothendieck topology on $\mathbb Z^{+}_D$ is given by\linebreak $J_{dis}(n) = \mathcal D(\mathfrak D_n) $.
\item The atomic Grothendieck topology on $\mathbb Z^{+}_D$ can only be defined if $\mathbb Z^{+}_D$ is downwards
directed, and is given by $J_{atom}(n) = \mathcal D(\mathfrak D_n)- \{\emptyset\}$.
\item A subset $D\subseteq \mathfrak D_n$ is said to be dense below $n$ if for any $m\leqslant n$ there is $k\leqslant m$ with $k\in D$. The dense topology on $\mathbb Z^{+}_D$ is given by $J(n)= \{ D\mid k\leqslant n\,\ \text{for all}\,\ k\in D,\,\ \text{and $D$ is a sieve dense below}\,\ n\}$ 
\end{itemize}
\end{ex}
\begin{defi}
A pair $(\mathbb Z^{+}_D, J)$, where $J$ is a Grothendieck topology on $\mathbb Z^{+}_D$, is called a site.
\end{defi} 
\subsection{Grothendieck topologies and its sheaves}
\begin{defi}
Let $F: (\mathbb Z^{+}_D)^{op} \longrightarrow Sets$ be a functor (a contravariant functor to Sets is also called a presheaf ). Let $n\in \mathbb Z^{+}$   and $S \in J(n)$. Then a family $ \left\langle a_k\right\rangle _{k\in S}$ in $\prod_{k\in S} F(k)$  is called a matching family for the cover $S$ with elements of $F$ if
\be
F(a_k) = a_t 
\ee
for each $k,t\in S$ such that $t \leqslant k$.

An element $a \in F(n)$  such that $F(a) = a_k$ for  each $k \in S$ is called an amalgamation. 

We say that $F$ is a $J$-sheaf if for each $n\in \mathbb Z^{+}$, for each $S \in J(n)$ and for each matching family $ \left\langle a_k\right\rangle _{k\in S}$ there is a unique amalgamation $a \in F(n)$.
\end{defi}
$Sh(\mathbb Z^{+}_D, J)$ will denote the category of all sheaves $F$ of sets  on the site ($\mathbb Z^{+}_D,J)$.
\begin{prop}
For  the site $(\mathbb Z^{+}_D,J_{tri})$ we have that $Sh(\mathbb Z^{+}_D) \cong Sets^{(\mathbb Z^{+}_D)^{op}}$
\end{prop}
\begin{proof}
Let $F$ in  $Sets^{(\mathbb Z^{+}_D)^{op}}$ and $n$ in $\mathbb Z^{+}_D$. Then there is only one cover of $n$, namely $\mathfrak D_n$. Hence if $ \left\langle a_k\right\rangle _{k\in \mathfrak D_n}$ is a matching family for $\mathfrak D_n$, we have $n\in \mathfrak D_n$ so $a_k= F(a_n)$. Thus $a_n$ is an amalgamation of the matching family. It is also unique, since if $a \in F(n)$ is another amalgamation then $a_k=F(a)=a$.
\end{proof}
\begin{prop}
 The category  $Sh(\mathbb Z^{+}_D)$ is equivalent to $1$ (the terminal object in $Sets^{(\mathbb Z^{+}_D)^{op}}$) for  the site $(\mathbb Z^{+}_D,J_{dis})$ .
\end{prop}
\begin{proof}
It is clear that $\emptyset \in J_{dis}(n)$ for each $n\in \mathbb Z^{+}$. Given a sheaf $F$ and a point $n\in \mathbb Z^{+}$, since a matching family for $\emptyset$ must be a function on the empty set, there exists only one function with the empty set as domain, and it is clearly a matching family. An amalgamation must be a point in $F(n)$, but since there are no restrictions
on this point except that it must be unique, it follows that $F(n)$ is a singleton set $1$. 
\end{proof}
\section{Subobject classifiers for topoi}
Perhaps the most important property for categories is the existence of a  subobject classifier $\Omega$. We will now show that there is such a subobject classifier for sheaves in the sites we have studied in previous sections.
\begin{defi}
We define $\Omega:(\mathbb Z^{+}_D)^{op}\longrightarrow Sets$ \ by
\be\label{sc}
\Omega(n)= \{ S\mid S\quad\text{is a principal sieve on}\quad n\}.
\ee
\end{defi}
\begin{prop}
The presheaf $\Omega$ of (\ref{sc}) is a sheaf.
\end{prop}
\begin{proof}
We must show that matching families have unique amalgamation in $\Omega$. There is only one cover of $n\in \mathbb Z^{+}$, namely $\mathfrak D_n$. Hence if $ \left\langle a_k\right\rangle _{k\in \mathfrak D_n}$ is a matching family for $\mathfrak D_n$, we have $n\in \mathfrak D_n$ so $a_k= F(a_n)$. Thus $a_n$ is an amalgamation of the matching family. It is also unique, since if $a \in F(n)$ is another amalgamation then $a_k=F(a)=a$.
\end{proof}
Observe that the maximal sieve on $n$,\,\ $t_n=\mathfrak D_n$ is obviously closed and that for $k\leqslant n$ in $\mathbb Z^{+}$, we have $t_k=\mathfrak D_k$. Thus $n\mapsto t_n$ defines a natural transformation 
\be\label{true}
true: 1\longrightarrow \Omega.
\ee
In this way
\begin{prop}
The sheaf $\Omega$, together with the map $true$ of (\ref{true}), is a subobject classifier of the category $Sh(\mathbb Z^{+}, J)$.
\end{prop}
\begin{proof}
Let $F$ be a sheaf on $\mathfrak D_n$ and $A\subseteq F$ be a subsheaf. It is enough to consider as ``characteristic function" the following:
$\chi_A:F\longrightarrow \Omega$ for $A$, defined by
\be
(\chi_A)_n(x)=\{ k\leqslant n\mid F_{\text{\tiny{$k\leqslant n$}}}(x)\in A(k)\}.
\ee
\end{proof}
As a conclusion, we have that

\centerline{$Sh(\mathbb Z^{+}_D, J)$  is an elementary topos.}

\section{ Some categories equivalent to $\mathbb Z^{+}_D$}
In this section we shall be concerned with various cases of partially ordered sets belonging to different areas of  mathematics and isomorphic to $(\mathbb Z^{+}_D, \leqslant)$ (i.e  they are  equivalent as categories).

\subsection{Periodic points of a map}\label{periodic}
We begin this section by reviewing the iterative process of discrete dynamical systems: let $f:X\rightarrow X$ be a map on a set $X$. The point $x\in X$ is a fixed point for $f$  if $f(x)=x$; the point $x\in X$ is a periodic point of period $n$ if $f^n(x)=x$\,\ ($f^n=f\circ f^{n-1},\ n\geq 1$). We denote the set of periodic points of period $n$ by $Per_n(f)$ and  the set of all periodic points of $f$ by $Per(f)$.

It is easy to show that  \,\ $Per_k(f)\subseteq Per_n(f)$\,\ if and only if $k$ is a divisor of $n$.  In this way, we will say that $Per_k(f)\leqslant  Per_n(f)$\,\ if and only if $k$ is a divisor of $n$ and we obtain that $(Per(f), \leqslant )$ is a partially ordered set which, considered as a category, has all binary products and all binary coproducts. In this way, if \,\ $Per_m(f),Per_n(f),Per_k(f) $ are objects of $Per(f)$ then $ Per_m(f)\leqslant Per_n(f)$ if and only if there is a unique arrow $Per_m(f)\rightarrow Per_n(f)$, the coproduct of $Per_m(f)$ and $Per_n(f)$  is $Per_k(f)$, where $k$=lcm\{m,n\}, that is $Per_m(f)\vee Per_n(f)=Per_{lcm\{m,n\}}(f) $ and the product is $Per_m(f)\land Per_n(f)=Per_{gcd\{m,n\}}(f)$. 

It is clear that $Per(f)$ is a distributive lattice.

Consequently,
\begin{prop}
The categories $Per(f)$ \ and\  $\mathbb Z^{+}_D$ are equivalent.
\end{prop}

\subsection{Roots of unity}\label{roots}
Now let us consider the roots of the unit in the field $\mathbb C$ of complex numbers. Remember that an element $z$\ of\ $\mathbb C$ is called a root of unity if there exists an integer $n \geqslant 1$ such that $z^n =1$. There are at most $n$ such roots, and they obviously form  a group $R_n$, which is cyclic. In forward, we will denote by $\mathcal R$ the set of all these groups $\{ R_n\mid n\in\mathbb Z^{+}\} $

It is well known that given two such groups $R_k$ and $R_n$, \,\ $R_k$ is a subgroup of $R_n$ if and only if $k$ is a divisor of $n$. Therefore, we will say that $R_k\leqslant  R_n$\,\ if and only if $k$ is a divisor of $n$ and we obtain that $(\mathcal R, \leqslant )$ is a partially ordered set which, considered as a category, has all binary products and all binary coproducts. In this case, if  $R_m, R_n, R_k $ are objects of $\mathcal R$ then $ R_m\leqslant R_n$ if and only if there is a unique arrow $R_m\rightarrow R_n$, the coproduct of $R_m$ and $R_n$  is $R_k$, where $k$=lcm\{m,n\}, in other words, $R_m\vee R_n=R_{lcm\{m,n\}} $ and the product is $R_m\land R_n=R_{gcd\{m,n\}}$. 

We also have that $Per(f)$ is a distributive lattice.
Consequently,
\begin{prop}
The categories $\mathcal R$\ and\  $\mathbb Z^{+}_D$ are equivalent.
\end{prop}
\begin{remark}
In the complex plane, the $n$ roots of unity correspond to the $n$ vertices of a regular $n$-sided polygon inscribed inside the unit circle, with one vertex at the point $z = 1$. The vertices are equally spaced on the circumference of the circle. The successive vertices are obtained by increasing the argument by an equal amount of $2\pi/n$ of the preceding vertex. Undoubtedly the collection of all such polygons make up another category equivalent to $\mathbb Z^{+}_D$.
\end{remark}
\subsection{Solutions of a differential equation}
 The combination of the results of \ref{periodic} y \ref{roots} allow us to analyze the topic of this subsection.

Let $I$ denote an open interval on the real line. The set of all complex-valued functions having $n$ continuous derivatives in $I$  is denoted by $\mathcal C^n(I)$. We are interested in finding all the solutions in $\mathcal C^n(I)$ of the differential equation
\be
\frac{d^n f}{dt}= f, \qquad \qquad f\in \mathcal C^n(I)
\ee
or writing in another way 
\be
\mathfrak D^n f= f,\qquad \qquad f\in \mathcal C^n(I).
\ee
That is to say, we are interested in finding the periodic points of the differential operator $\mathfrak D$.

For $n=1$, we have that $\mathfrak D f=f$ if and only if $f(t)= ke^t,\quad k\in \mathbb R$. In this way $V_1= \langle e^t \rangle=\{ke^t,\mid k\in \mathbb R\}$ is the real vector space of all solutions of $\mathfrak D f=f$.

In general, if $f(t)= e^{\omega t}$ then $\mathfrak D^n f(t)=\omega^n e^{\omega t}$; in consequence,\linebreak $\mathfrak D^n f(t)= f(t)$ if and only if $\omega^n=1$. i.e. $\omega$ is a root of the unit in the field $\mathbb C$ of complex numbers, and therefore,   $V_n= \langle e^t, e^{\omega t}\cdots, e^{\omega^{(n-1)} t} \rangle$ is the real vector space of all solutions of $\mathfrak D^n f=f$. In this paper, we will denote by $\mathcal V$ the set of all these vector spaces $\{ V_n\mid n\in\mathbb Z^{+}\} $
It is clear that $V_k$ is a subspace of $V_n$ iff $k\in \mathfrak D_n$. For this reason we will say that $V_k\leqslant  V_n$\,\ if and only if $k$ is a divisor of $n$ and we obtain that $(\mathcal V, \leqslant )$ is a partially ordered set which, considered as a category, has all binary products and all binary coproducts. In this case, if  $V_m, V_n, V_k $ are objects of $\mathcal V$ then $ V_m\leqslant V_n$ if and only if there is a unique arrow $V_m\rightarrow V_n$, the coproduct of $V_m$ and $V_n$  is $V_m\vee V_n=V_{lcm\{m,n\}} $ and the product is $V_m\land V_n=V_{gcd\{m,n\}}$. 
Thus,
\begin{prop}
The categories $\mathcal V$\ and \ $\mathbb Z^{+}_D$ are equivalent.
\end{prop}
Note, for example,  that a Grothendieck topology  on  the category $\mathcal V$ is a function $J$ which assigns to each object $V_n$\,\ of\,\  $\mathcal V$ a collection $J(n)$ of sieves on $\mathcal V$, in such a way that
\begin{enumerate}
\item[(i)] the maximal sieve $\downarrow V_n=\{ V_k\mid k\in  \mathfrak D_n\}$ is in $J(n)$;
\item[(ii)] (stability axiom) if $S\in J(n)$ and $V_k\leqslant V_n$ then $S\cap \downarrow V_k$ is in $J(k)$;
\item[(iii)] (transitivity axiom) if $S\in J(n)$ and $R$ is any sieve on $V_n$ such that $R\cap\downarrow  V_k$ is in $J(k)$ for each $k\in S$, then $R\in J(n)$.
\end{enumerate}

Our final suggestion is to find in  elementary mathematics other categories equivalent to the previous ones.


\begin{thebibliography}{10}
\bibitem{GB}{\sc G. Birkhoff}, { \it Lattice Theory}, American Mathematical Society, Providence, 1940. 
\bibitem{PJ}{\sc P. T. Johnstone}, {\it Sketches of an Elephant A Topos Theory Compendium}, Vol 1,2,  Oxford University Press, Cambridge, 2002.
\bibitem{AL}{\sc A.J. Lindenhovius}, {\it Grothendieck topologies on posets}, arXiv:1405.4408v2, 2014.
\bibitem{SM}{\sc S. MacLane}, {\it Categories for the Working Mathematician}, Springer-Verlag, New York / Heidelberg / Berlin,1971.
\bibitem{MM} {\sc S. MacLane and  I. Moerdijk}, {\it Sheaves in Geometry and Logic,{ \scriptsize A first introduction to Topos theory}}, Springer-Verlag, New York / Heidelberg / Berlin,1992.
\
\end{thebibliography}
\end{document}